\documentclass[12pt]{article}
\usepackage[T2A]{fontenc}
\usepackage[cp1251]{inputenc}
\usepackage[english,russian]{babel}
\usepackage{graphicx}
\usepackage{amsmath}
\usepackage{color}
\usepackage{amssymb}
\usepackage{chicago}
\usepackage{graphicx}
\usepackage{epstopdf}
\usepackage{epsfig}
\usepackage{multirow}

\newtheorem{theorema}{Теорема}

\newtheorem{claim}{Утверждение}
\newtheorem{ddef}{Определение}

\newtheorem{example}{Пример}
\newtheorem{lemma}{Лемма}

\newtheorem{lemm_conseq}{Следствие}

\newenvironment{proof}{\par\noindent{\bf Доказательство.}}{\hfill$\scriptstyle\blacksquare$}
\makeatletter
\renewcommand \thesection {\@arabic\c@section.}
\renewcommand\thesubsection {\thesection\@arabic\c@subsection.}
\renewcommand\thesubsubsection {\thesubsection\@arabic\c@subsubsection.}
\renewcommand\theparagraph {\thesubsubsection\@arabic\c@paragraph.}

\begin{document}
\selectlanguage{russian}

\title{
{\small \begin{flushright}
УДК 519.714
\end{flushright}}
Нижние оценки сложности покрытия почти всех точек булева куба гранями и смежные задачи
}

\author{Ю.~В. Максимов\thanks{Исследование выполнено при финансовой поддержке РФФИ в рамках научного проекта \No\,14-07-31277 мол\_а; а также при частичной поддержке Лаборатории структурных методов анализа данных в предсказательном моделировании ФУПМ МФТИ, грант правительства РФ дог. 11.G34.31.0073.
}\\
ПреМоЛаб МФТИ и ИППИ РАН}
\date{}

\maketitle

\begin{abstract}
В настоящей работе рассматриваются задача покрытия гранями всех точек булева куба, кроме не более чем логарифмического по размерности множества заданных, которые покрыты быть не должны. Указанная задача возникает в некоторых задачах машинного обучения, сложности и построения оптимальных булевых схем.

В работе решается задача определения конфигурация точек, для покрытия которых требуется наиболее сложное в различных смыслах множество граней.

\noindent
\textbf{Ключевые слова:}
покрытие множеств, булева функция, дизъюнктивная нормальная форма, нижние оценки сложности.
\end{abstract}

\section{Введение}
В настоящей работе рассматриваются задача покрытия гранями всех точек булева куба кроме заданной конфигурации точек. Наибольший интерес, мотивированный практическими приложениями, при этом представляет случай, когда число непокрытых точек достаточно мало.

Указанная задача возникает в ряде задач машинного обучения, дискретной математики и комбинаторной оптимизации \cite{boros11,zh78}.

Задача может быть переписана в эквивалентной форме с точки зрения построения простых дизъюнктивнх нормальных форм (ДНФ) булевых функций. А именно, в настоящей работе исследуются булевы функции, заданные в конъюнктивной нормальной форме (КНФ) произведениями вида
\begin{gather}\label{general_eq}
    f(x_1,x_2,...,x_n) = \bigwedge\limits_{i=1}^k \left( x_1^{\alpha_i^1} \vee x_2^{\alpha_i^2} \vee ...\vee x_n^{\alpha_i^n}\right).
\end{gather}

Рассматриваемая задача представляет собой частный случай задачи исключения нулевых точек из ДНФ (DNF Exception problem). Указанная задача была подробно рассмотрена в работах \cite{mu06,bp92}. Проблема исключения нулевых точек из ДНФ может быть сформулирована следующим образом: ДНФ $D$ из $m$ пересекающихся конъюнкций необходимо преобразовать к ДНФ минимальной сложности таким образом, что бы полученная ДНФ $D'$ совпадала с $D$ во всех точках кроме заданных $k$, на которые $D'$ должна обращаться в ноль. В данной работе мы будем рассматривать только оценку числа литералов в ДНФ булевой функции и единственной исходной конъюнкции тождественно равной единице. В настоящей работе нас будут интересовать только нижние оценки на ранг полученной ДНФ для случая логарифмического по размерности числа нулей $k$.

Целью работы является в поиске таких функций, сложность представления которых в классе ДНФ будет наибольшей.

Задача минимизации сложности функции \ref{general_eq} состоит в том, что бы записать указанную функцию в дизъюнктивной нормальной форме (ДНФ) с использованием как можно меньшего числа логических произведений (конъюнкций), а также символов переменных и их отрицаний (литералов). \emph{Длиной} ДНФ называется число входящих в нее конъюнкций, \emph{рангом} ДНФ называется число использованных в её записи литералов. Дизъюнктивная нормальная форма минимального ранга называется \emph{минимальной}, а минимальной длины --- \emph{кратчайшей}. С точки зрения построения оптимальных покрытий, ДНФ минимального длины соответствует покрытие использующее минимальное число граней; ДНФ минимального ранга соответствует минимальное вес совокупности граней, где в качестве веса грани выступает разность между размерностью функции и размерностью грани.

Задача минимизации ДНФ булевой функции с заданным множеством нулевых точек может быть рассмотрена как частный случай $NP$-трудной задачи о покрытии множеств, при котором множество единиц функции следует покрыть наиболее простым множеством граней булева гиперкуба (hypercube covering problem).

В общем случае, для произвольных $k$ задача восстановления минимальной (кратчайшей) ДНФ булевой функции по заданному представлению \ref{general_eq} $NP$-полна \cite{mas79}. Более того, задача определения по заданной ДНФ является ли она минимальной является еще более сложной переборной задачей, а именно принадлежит классу $\Sigma_2^p$ \cite{um98}.

Тем не менее, при небольших значениях $k$, указанная задача в ряде случаев может быть решена точно. В частности точное решение при $k=2$ для всякой функции $f$ было найдено С. В. Яблонским, при $k=3,4$ А.~Ю.~Коганом и Ю.~И.~Журавлёвым. Кроме того, ими же в конце 1980-ых годов было замечено, что при $k\le \log_2 - \log_2\log_2 n + 1$ почти всем булевым функциям от $n$ переменных, при $n\rightarrow\infty$, может быть сопоставлена некоторая функция $\phi_k$, полностью определяющая сложность минимальных и кратчайших ДНФ этих функций \cite{zk85dan}. В той же работе указанная функция была названа \emph{полной}.

До недавнего времени оставался открытым вопрос, является ли функция $\phi_k$ наиболее сложной в классе функций от того же числа переменных, принимающих значение $0$ ровно в $k$ точках.

Частичный, ответ на этот вопрос был анонсирован автором в кратком сообщении \cite{me12b}.

В настоящей работе приводиться усиление результата \cite{me12b}, позволяющее устанавливать более точные границы на сложность ДНФ булевых функций, обращающиеся в 0 в небольшом числе точек. Предлагаемый метод основан на идее анализа покрытия гранями куба не всех точек, на которых рассматриваемая функция функция принимает значение 1, а только их небольшой части ``трудной'' для покрытия множествами указанной формы.
Указанный способ позволяет существенно усилить классические Шенноновские оценки сложности ДНФ и является одним из первых, позволяющих аналитически установить явные нижние оценки сложности для ДНФ конкретных функций.

\section{Предварительные результаты, обозначения, терминология}

Основным объектом рассмотрения является матрица нулей $M_f$ булевой функции $f(x_1,x_2,...,x_n)$, по строкам которой последовательно выписаны все различные нули функции $f$. В настоящей работе получены нижние оценки ранга булевых функций, матрицы нулей которых не содержат столбцов с небольшим числом нулей или единиц. Указанные оценки оказываются полезными при исследовании булевых функций от большого числа переменных, имеющих небольшое число нулевых точек.

Пусть $\pi_n$ --- класс (инвариантных) преобразований Шеннона-Поварова булевых функций от $n$ переменных таких, что всякое преобразование $\pi \in \pi_n$ сопоставляет упорядоченному набору переменных $(x_1,x_2,...,x_n)$ набор $(y_1,y_2,...,y_n) = (x_{i_1}^{\sigma_{i_1}},x_{i_2}^{\sigma_{i_2}},...,x_{i_n}^{\sigma_{i_n}})$, где $(i_1,i_2,...,i_n)$ --- произвольная перестановка множества $\{1,2,...,n\}$, обозначаемого в дальнейшем $[n]$, $\sigma_i \in \{0,1\}$ при $i\in[n]$. Отметим, что булева функция и ее образ при преобразовании Шеннона-Поварова имеют одинаковую сложность при представлении их дизъюнктивными нормальными формами.

Обозначим $P^n_k$ класс булевых функций от $n$ переменных, имеющих в точности $k$ различных нулей. Произвольной функции $f \in P_k^n$ сопоставим функцию $g \in P_k^l$ таким образом, что
\begin{enumerate}
\item   в матрице $M_f$ отсутствуют нулевой и единичный столбцы;
\item   одинаковые столбцы в $M_g$ расположены последовательно;
\item   для любой пары столбцов, один из которых является отрицанием другого, в матрице $M_g$ присутствует не более одного.
\end{enumerate}

Заметим, что инвариантными преобразованиями Шеннона-Поварова всякая булева функция, матрица нулей которой не имеет нулевых (единичных) столбцов, может быть записана в таком виде. При этом сложность ДНФ функции не меняется. Булева функция $g$ называется \textit{правильной}, если ее матрица нулей удовлетворяет условиям~$1-3$. Если кроме того матрица нулей $g$ не содержит одинаковых столбцов, то $g$ называется \textit{приведенной}.

Приведенную функцию, принадлежащую классу $P_k^{2^{k-1}-1}$, назовем \textit{полной}. 
Согласно результату, полученному в работах \cite{zk85dan,zk86}, всякая правильная булева функция $\phi$ представима в виде
\begin{eqnarray}\label{formula_reduction}
    D_{\phi}(x_1, x_2, ..., x_n) = D_F(x_{i_1}, x_{i_2}, ...x_{i_t}) \vee D_2(x_{i_1}, x_{i_1+1}, ..., x_{i_2-1}) \vee D_2(x_{i_t}, x_{i_t+1}, ..., x_{i_{t+1}-1})
\end{eqnarray}
где $i_1 = 1$, а $i_{t+1} = n+1$; столбцы с номерами $i_s, i_{s+1}, ..., i_{s+1}-1$ равны при $1\le s \le t$; набор столбцов с номерами $i_1,i_2, ..., i_t$ состоит в точности из $t$ различных столбцов. Здесь $D_F $ в формуле (\ref{formula_reduction}) --- ДНФ соответствующей $\phi$ приведенной функции $F$, а $D_2(x_{i_s}, x_{i_{s+1}}, ..., x_{i_{s+1} - 1}) = x_{i_s} \bar{x}_{i_s + 1} \vee x_{i_s+1} \bar{x}_{i_s + 2} \vee ... \vee x_{i_{s+1}-1} \bar{x}_{i_s}$ при $1\le s \le t$. Отметим, что указанным способом можно сводить построение ДНФ произвольной булевой функции к построению ДНФ соответствующей приведенной функции,  при этом сложность сведения линейна.
\begin{theorema}
Почти всем правильным булевым функциям от $n$ переменных, которые имеют не более $k < \log_2 n - \log_2 \log_2 n + 1$ нулей, при $n\rightarrow\infty$ по формуле \ref{formula_reduction} ставиться в соответствие полная булева функция.
\end{theorema}

Несмотря на это, при числе нулей большем чем $\log_2 n - \log_2\log_2 n + 1$ большинству функций сопоставляются приведенные функции, отличные от полных.

Со времени появления теории сложности дизъюнктивных нормальных форм с малым числом нулей \cite{zk85dan,zk86}, оставался открытым вопрос является ли полная функция (асимптотически) экстремальной по сложности в своем классе. 

В настоящей работе показано, что при рассмотрении в качестве меры сложности числа литералов, содержащихся в ДНФ, указанная открытый вопрос разрешается отрицательно. А именно,
в классе $P_k^{2^{k-1}-1}$ существует булева функция, минимальная ДНФ реализация которой содержит более $5\cdot\frac{2^k}{3}(1+o(1))$ литералов при $k\rightarrow\infty$. В то же время в работе \cite{me12b} было установлено, что число литералов в минимальной ДНФ полной функции не превосходит $3\cdot2^{k-1}(1+o(1))$ при~$k\rightarrow\infty$.

\section{Определения и обозначения}
В данной работе использован ряд стандартных терминов теории булевых функций, определенных в работах \cite{ja10,dj01}. Кратко напомним основные из них.


Обозначим $M_{i_1,...,i_t}^{j_1,...,j_s}$ подматрицу матрицы $M$, находящуюся на пересечении строк $i_1,...,i_t$ со столбцами $j_1,...,j_s$. Матрицу, не содержащую одинаковых строк, назовем \textit{тестом}.

Литералом будем называть булеву переменную или ее отрицание. Литерал $x_i$ булевой функции $f(x_1,x_2,...,x_n)$ ассоциируем со столбцом $M_{[k]}^i$ матрицы нулей функции $f$; литерал $\bar{x_i}$ ассоциируем с покоординатным отрицанием столбца $M_{[k]}^i, 1\le i \le n$. Далее под термином литерал $x_i^{\sigma_i}$ будет часто подразумеваться ассоциированный с ним булев вектор.

Положим $B_k$ --- множество всех булевых векторов размерности $k$; $B_k^0$ и $B_k^1$ --- множество булевых векторов, первая координата которых равна 0 или 1 соответственно; под вектором $\bar{x}$ будем понимать покоординатное отрицание вектора $x\in B_k$; минимум из числа единиц и числа нулей $x$ назовем \textit{весом} вектора $x, x\in B_k$.

Используем следующие понятия, определенные в работе \cite{me12b}:
\begin{ddef}
Ненулевой вектор $\alpha \in B_k$ назовем \textit{разложимым} по векторам $\alpha_1,\alpha_2, ...,\alpha_t$ если выполнены следующие равенства
\[
\alpha = \alpha_1 \vee \alpha_2 \vee ... \vee \alpha_t, \qquad <\!\bar{\alpha},\alpha_1\!> = <\!\bar{\alpha}, \alpha_2\!> = ... = <\!\bar{\alpha},\alpha_t\!> = 0,
\]
где под символом $<\!\alpha, \beta\!>$ понимается скалярное произведение векторов $\alpha$ и $\beta$.
\end{ddef}
\begin{ddef}
Ненулевой вектор $\alpha \in B_k$ назовем \textit{ортогонально разложимым} по векторам $\alpha_1, ...,\alpha_t$ если выполнены следующие равенства
\[\begin{cases}
\alpha = \alpha_1 \oplus \alpha_2 \oplus ... \oplus \alpha_t\\
\alpha = \alpha_1 \vee \alpha_2 \vee ... \vee \alpha_t
\end{cases}\]
\end{ddef}

Обозначим $\mathbb{I}_k$ вектор размерности $k$, состоящий только из единиц. В случае $\alpha = \mathbb{I}_k$ в условиях предыдущего определения назовем вектора $\{\alpha_i\}_{i=1}^t$ \textit{ортогональным разложением единицы}. Заметим, что если вектор $\alpha$ разложим по векторам $\{\alpha_i\}_{i\in I}$, то $\chi(\alpha) > \chi(\alpha_i), i\in I$.

Конъюнкция $K$ называется \textit{импликантой} функции $f(x_1,x_2,...,x_n)$, если $N_K \subseteq N_f$. Импликанта $K$ называется простой, если из $K$ не может быть вычеркнут ни один литерал так, чтобы полученная конъюнкция была импликантой $f(x_1,x_2,...,x_n)$. Термины \textit{простая импликанта} и \textit{несократимая конъюнкция} употребляются в работе как синонимы.

Обозначим $Z_f$ --- множество нулей функции $f$, а $N_f$ --- множество ее единиц. Определим вектора $\{e_i^k\}_{i=1}^{k}$ равенствами $e_{i+1}^k = \chi^{-1}(2^i)$, $0~\le~i~<~k$.

Пусть $E(t) = \{i:M_i^t = 1\}$, $Z(t) = [k]\setminus E(t)$. Отметим, что для приведенной функции $Z(t)\neq\varnothing$ и $E(t)\neq\varnothing$. Следуя работе \cite{dj01}, обозначим $\tilde{\theta}(i,j) = (\theta_1,\theta_2,...,\theta_n)$ точку булева куба такую, что $\theta_r = M_i^r$ при всех $r\in[n]\setminus \{j\}$ и $\theta_j = 1 - M_i^j$.

Далее будем полагать, что число единиц каждого столбца матрицы нулей всякой рассматриваемой далее функции не превосходит числа нулей. Отметим, что всякая булева функция может быть приведена к такому виду преобразованиями Шеннона-Поварова.

Для конъюнкции $K$ обозначим $rank^{+}\; K$ число положительных литералов конъюнкции; $rank^{-}\; K$ число отрицательных литералов; \[rank\;K = rank^{+}\;K + rank^{-}\;K.\]

В работе \cite{dj01} введено следующее определение
\begin{ddef}
Множество
\[\bigcup\limits_{i=1}^k\bigcup\limits_{j=1}^n \{\tilde\theta(i,j)\}\setminus \bar{N}_f\]
назовем \textit{множеством околонулевых точек}.
\end{ddef}

Обозначим указанное множество $\Theta$. Положим
\[\Theta^{1} = \bigcup\limits_{i=1}^k\bigcup_{\substack{M_i^j = 1\\ j\in \{1,2,...,n\}}} \{\tilde\theta(i,j)\}\setminus \bar{N}_f
\qquad \Theta^{0} = \bigcup\limits_{i=1}^k\bigcup_{\substack{M_i^j = 0\\ j\in \{1,2,...,n\}}} \{\tilde\theta(i,j)\}\setminus \bar{N}_f\]

А.~Г.~Дьяконовым в той же работе была доказана
\begin{lemma}
Для всех элементарных конъюнкций $K\in D^{\text{сокр}_f}$ ($D^{\text{сокр}_f}$ --- сокращенная ДНФ функции $f$) и всех $i \in [k]$ справедливо
\[\left|N_K \cap \bigcup_{j=1}^n \tilde\theta(i,j)\right| \le 1\]
\end{lemma}

\section{Нижняя оценка сложности}
\textit{Соседними} назовем булевы вектора одинаковой размерности, находящиеся на расстоянии 1 в метрике Хемминга. Рассмотрим произвольную приведенную булеву функцию $f(x_1,...,x_n)$, не имеющую соседних нулевых точек; обозначим $M_f$ ее матрицу нулей и зафиксируем ее произвольную ДНФ $D$.

Будем считать, что всякая рассматриваемая далее конъюнкция несократимая, то есть является простой импликантой $f$. Аналогично будем полагать, что всякая рассматриваемая в этом параграфе ДНФ состоит исключительно из простых импликант. Кроме того, будем полагать, что каждая рассматриваемая далее матрица является тестом.

Определим класс приведенных булевых функций $\Phi^{n,k}$ как множество функций от $n$ переменных, имеющих $k$ нулей, никакие два из которых не являются соседними. Скажем, что функция $f$ принадлежит классу $\Phi^{n,k}_\lambda$, если $f\in \Phi^{n,k}$ и кроме того, все столбцы, составляющие матрицу нулей $M_f$, имеют вес не меньше $\lambda$. Кроме того, будет считать, что число единиц в каждом столбце матрицы $M_f$ не превосходит числа нулей.

Введем важное для последующих рассуждений определение из работы \cite{me12b}
\begin{ddef}
\textit{Разрезанием} теста $M$ на $t, t\ge1$ частей, назовем совокупность непустых подматриц $\{M_i\}_{i=1}^t$ таких, что
\begin{enumerate}
\item Матрицы $M_1, M_2,...,M_t$ имеют одинаковое число столбцов;
\item каждая строка матрицы $M$ присутствует хотя бы в одной матрице $M_j, 1\le j \le t$;
\item число строк матрицы $M$ совпадает с суммой числа строк в матрицах $\{M_i\}_{i=1}^t$.
\end{enumerate}
\end{ddef}

\begin{lemma}\label{literal_1}
Всякий литерал функции $f$ входит не менее чем в одну конъюнкцию.
\end{lemma}
\begin{lemma}\label{literal_coverage}
Пусть литерал $x_i^{\sigma_i}$ разложим по литералам $x_{i_1}^{\sigma_{i_1}\oplus1}, x_{i_2}^{\sigma_{i_2}\oplus1},...$, $x_{i_t}^{\sigma_{i_t}\oplus1}$.
Пусть $\beta = (\beta_1,\beta_2,...,\beta_n)$ --- такая точка булева куба, что $\beta_i = \sigma_i, \beta_{i_1} = \sigma_{i_1}, \beta_{i_2} = \sigma_{i_2},..., \beta_{i_t} = \sigma_{i_t}$, тогда конъюнкция $K~=~x_i^{\sigma_i}x_{i_1}^{\sigma_{i_1}}x_{i_2}^{\sigma_{i_2}}\cdots x_{i_t}^{\sigma_{i_t}}$ является допустимой для $f$ и, более того, $K[\beta] =1$.
\end{lemma}

Конъюнкцию $K$ в условиях леммы назовем конъюнкцией определяемой разложением $x_i^{\sigma_i}$ по векторам $x_{i_1}^{\sigma_{i_1}\oplus1}, x_{i_2}^{\sigma_{i_2}\oplus1},...,x_{i_t}^{\sigma_{i_t}\oplus1}$. В работе \cite{me12b} была доказана ключевая для дальнейших рассуждений лемма о \textit{числе минимальных литералов}

Рассмотрим приведенную булеву функцию $f(x_1,...,x_n)$ заданную матрицей нулей $M_f$ и не имеющую смежных нулевых точек. Зафиксируем ее литерал $x_i^{\sigma_i}, \sigma_i \in\{0,1\}$ и некоторую реализующую ее ДНФ $D$. Справедливы следующие утверждения, доказательства которых приведены далее

\begin{lemma}\label{literal_min_number_a}
Выделим в $D$ конъюнкции, содержащие литерал $x_i^{\sigma_i}$ в количестве $t$ штук. Литерал $x_i^{\sigma_i}$ входит еще как минимум в одну, отличную от выделенных, конъюнкцию из $D$, если при любом разрезании матрицы $M_f$ на $t$ частей $M_1,M_2,...,M_t$ существует такая матрица $M_j$, что для функции $\phi_j$ заданной матрицей $M_j$, ни одна из выделенных конъюнкций не определяет разбиение $x_i^{\sigma_i}$.
\end{lemma}
\begin{lemma}\label{literal_min_number_b}
Литерал $x_i^{\sigma_i}$ входит не менее чем в $t+1$ конъюнкцию каждой ДНФ функции $f$, если при любом разрезании матрицы $M_f$ на $t$ частей $M_1,M_2,...,M_t$, существует такая матрица $M_j$, что для функции $\phi_j$, заданной матрицей $M_j$, не существует набора литералов, образующего разбиение литерала $x_i^{\sigma_i}$.
\end{lemma}
Кроме того, заметим что если при разрезании матрицы $M$ на части $M_1,M_2,...,M_t$ для каждой из функций, определяемых матрицами $M_1,M_2,...,M_t$, существует (свое для каждой части) разбиение литерала $x_i^{\sigma_i}$, порождающее при этом допустимую конъюнкцию $f(x_1,x_2,...,x_n)$, то существует ДНФ $\hat D$ функции $f$, такое что литерал $x_i^{\sigma_i}$ входит не более чем в $t$ различных конъюнкций содержащихся в $\hat D$.

\begin{lemm_conseq}
Литерал $x_i^{\sigma_i}$ входит не менее чем в две конъюнкции произвольной ДНФ~$D$, если отрицания литералов, содержащихся в конъюнкции $K\setminus x_i^{\sigma_i}$, при $K\ni x_i^{\sigma_i}$, не образуют разбиение~$x_i^{\sigma_i}$.
\end{lemm_conseq}
\begin{ddef}
Литерал $x_i^{\sigma_i}$ назовем \textit{собственным} литералом конъюнкции $K$ дизъюнктивной нормальной формы $D$ функции $f(x_1,...,x_n)$, если $K$ --- единственная конъюнкция $D$, содержащая $x_i^{\sigma_i}$.
\end{ddef}


Рассмотрим фильтрацию $\Phi^{n,k}_{k/2} \subseteq \Phi^{n,k}_{k/2-1} \subseteq \Phi^{n,k}_{k/2-2} \subseteq ... \subseteq \Phi^{n,k}_{[k/3]+1}$. Каждое из последующих утверждений, доказанное для функций класса $\Phi^{n,k}_{m}$, будет справедливо для всех непустых вложенных классов фильтрации.

Докажем следующие утверждения о структуре конъюнкций входящих в ДНФ реализации функций из класса $f\in \Phi^{n,k}_{[k/3]+1}$. Для определенности зафиксируем функцию $f$. Из того, что функция $f$ приведенная следует
\begin{claim}
Конъюнкции ранга два не имеют собственных литералов.
\end{claim}

\begin{claim}\label{conj_structure_own_plus}
Всякая конъюнкция, 
имеющая один собственный положительный литерал $x_i$, представима в виде $x_i K_a$, где $K_a$ --- антимонотонная конъюнкция ранга не менее 2, отрицания литералов которой образуют разбиение $x_i$.
\end{claim}
\begin{proof}
Рассмотрим конъюнкцию $K = x_i x_{i_1}^{\sigma_{i_1}}x_{i_2}^{\sigma_{i_2}}\cdots x_{i_t}^{\sigma_{i_t}}$. Так как литерал $x_i \in K$ собственный, то литералы $x_{i_1}^{1 - \sigma_{i_1}}$, $x_{i_2}^{1 - \sigma_{i_2}}$, ...,$x_{i_t}^{1 - \sigma_{i_t}}$ образуют разложение $x_i$. В силу того, что $f$ приведенная и число единиц во всяком столбце матрицы нулей $f$ не превосходит числа нулей, то в разложение $x_i$ могут входить только положительные литералы. Следствием этого является утверждение леммы.
\end{proof}

Конъюнкции представимые в виде $x_i K_a$, где $K_a$ --- некоторая антимонотонная конъюнкция, такая что $x_iK_a$ не имеет собственных литералов, кроме, быть может, $x_i$, назовем конъюнкциями класса $\mathbb{K}_1$.

\begin{claim}\label{conj_structure_own_minus}
Всякая конъюнкция, 
имеющая один собственный отрицательный литерал $\bar{x}_i$, представима в одной из следующих форм
\begin{enumerate}
\item[$a.$] $K = \bar{x}_i \bar{x}_{i_1}\bar{x}_{i_2}\cdots\bar{x}_{i_{t-1}}$;
\item[$b.$] $K = \bar{x}_i {x}_{i_1}\bar{x}_{i_2}x_{i_3}^{\sigma_{i_3}}x_{i_4}^{\sigma_{i_4}}\cdots x_{i_{t-1}}^{\sigma_{i_{t-1}}}$, $\sigma_{i_j}\in\{0,1\}$;
\item[$c.$] $K = \bar{x}_i {x}_{i_1}{x}_{i_2}x_{i_3}^{\sigma_{i_3}}x_{i_4}^{\sigma_{i_4}}\cdots x_{i_{t-1}}^{\sigma_{i_{t-1}}}$, $\sigma_{i_j}\in\{0,1\}$,
\end{enumerate}
при этом отрицания литералов из $K\setminus \bar{x}_i$ образуют разложение $\bar{x_i}$, а $t\ge 3$
\end{claim}
\begin{proof}
Фактически в лемме утверждается отсутствие конъюнкций ранга два, имеющих собственные литералы, что справедливо в силу того, что $f$ приведенная.
\end{proof}

Конъюнкции, представимые в виде $a$, определяемом леммой \ref{conj_structure_own_minus}, назовем конъюнкциями класса $\mathbb{K}_2$; класс $\mathbb{K}_3$ составляют конъюнкции, представимые в виде $b$, но не представимые в виде $a$; класс $\mathbb{K}_4$ состоит из конъюнкций, имеющих один собственный отрицательный литерал, но не принадлежащих ни классу $\mathbb{K}_2$, ни классу $\mathbb{K}_3$.

\begin{claim}\label{conj_structure_rank2}
Всякая конъюнкция ранга два 
представима в виде ${x}_i{x}_{j}$, или $x_i \bar{x}_{j}$ для некоторых $i,j$ и не имеет собственных литералов.
\end{claim}
\begin{proof}
Собственные литералы в конъюнкции ранга 2 отсутствуют так как $f$ приведенная; конъюнкции ранга два вида $\bar{x}_i \bar{x}_j$ не являются импликантами $f$ так как покрывают один из ее нулей.
\end{proof}

Конъюнкции, представимые в виде ${x}_i{x}_{j}$, составляют класс $\mathbb{K}_5$, а в виде $x_i\bar{x}_j$  --- класс $\mathbb{K}_6$.
\begin{lemma}
Всякая конъюнкция произвольной ДНФ реализации функции $f\in\Phi^{n,k}_{[k/3]+1}$ лежит в одном из классов $\mathbb{K}_1,\mathbb{K}_2,...,\mathbb{K}_6$.
\end{lemma}
\begin{proof}
Так как вес каждого литерала строго больше $k/3$ то ни одна конъюнкция не имеет двух и более собственных литералов.
\end{proof}

Зафиксируем далее некоторую функцию $f \in \Phi^{n,k}_{m} \subseteq \Phi^{n,k}_{[k/3]+1}$, и ее ДНФ реализацию $D$.

\begin{claim}\label{conj_class_1}
Для всякой конъюнкции $K\in\mathbb{K}_1$ выполнены неравенства
\[
\begin{cases}
\left|\Theta^0 \cap N_K\right| \le 2\Delta \\
\left|\Theta^1 \cap N_K\right| \le k/2 + \Delta \\
rank^{-}(K) \ge 2 \\
rank^{+}(K) = 1,
\end{cases}\]
где $\Delta = k/2 - m$.
\end{claim}
\begin{proof}
По определению класса $\mathbb{K}_1$ всякая конъюнкция $K$ представима в виде $x_i \bar{x}_{i_1} \bar{x}_{i_2} \cdots \bar{x}_{i_t}$ для некоторого $t$ и некоторых литералов $x_i, \bar{x}_{i_1}, \bar{x}_{i_2}, ..., \bar{x}_{i_t}$. Так как $x_i$ при этом является собственным литералом, то литералы $x_{i_1}, x_{i_2}, ..., x_{i_t}$ образуют разбиение $x_i$. Максимальный вес литерала $x_i$ не превосходит $k/2$; минимальный $k/2 - \Delta$, откуда следуют ограничения на $\left|\Theta^0 \cap N_K\right|$ и $\left|\Theta^1 \cap N_K\right|$.
\end{proof}

Доказательства лемм \ref{conj_class_2} -- \ref{conj_class_4} основаны на аналогичных рассуждениях.
\begin{claim}\label{conj_class_2}
Для всякой конъюнкции $K\in\mathbb{K}_2$ выполнено
\[\begin{cases}
\left|\Theta^0 \cap N_K\right| \le k/2+3\Delta \\
\left|\Theta^1 \cap N_K\right| = 0 \\
rank^{-}(K) \ge 3\\
rank^{+}(K)  = 0,
\end{cases}
\]
где $\Delta = k/2 - m$.
\end{claim}
\begin{proof}
По определению класса $\mathbb{K}_2$ для некоторого $t$ и литералов $\bar{x}_i, \bar{x}_{i_1}, \bar{x}_{i_2}, ..., \bar{x}_{i_t}$ представима в виде $K = \bar{x}_i \bar{x}_{i_1}\bar{x}_{i_2}\cdots\bar{x}_{i_t}$, при этом литерал $\bar{x}_i$ собственный; следовательно ${x}_{i_1}, {x}_{i_2}, ..., {x}_{i_t}$ разложение вектора $\bar{x}_i$. Учитывая, что веса всех литералов ограничены снизу числом $k/2 - \Delta$, сверху числом $k/2$, получаем утверждение леммы.
\end{proof}

\begin{claim}\label{conj_class_3}
Для всякой конъюнкции $K\in\mathbb{K}_3$ выполнено
\[
\begin{cases}
    \left|\Theta^0 \cap N_K\right| \le k/2 \\
    \left|\Theta^1 \cap N_K\right| \le 2\Delta \\
    rank^{-}(K) \ge 2 \\
    rank^{+}(K) \ge 1,
\end{cases}
\]
где $\Delta = k/2 - m$.
\end{claim}

\begin{claim}\label{conj_class_4}
Для всякой конъюнкции $K\in\mathbb{K}_4$ выполнено
\[
\begin{cases}
\left|\Theta^0 \cap N_K\right| \le k/2\\
\left|\Theta^1 \cap N_K\right| \le \Delta\\
rank^{-}(K) \ge 1\\
rank^{+}(K) \ge 2,
\end{cases}
\]
где $\Delta = k/2 - m$.
\end{claim}

Следующие две леммы являются следствием того, что отрицания литералов конъюнкции входящей в ДНФ $D$ функции $f$, образуют разбиение $\mathbb{I}_k$. Кроме того, так как функция $f$ приведенная, то никакие два различных ассоциированных с разными столбцами $M_f$ не совпадают.
\begin{claim}\label{conj_class_5}
Для всякой конъюнкции $K\in\mathbb{K}_5$ выполнено
\[
\begin{cases}
    \left|\Theta^0 \cap N_K\right| = 0 \\
    \left|\Theta^1 \cap N_K\right| \le k-1\\
    rank^{+}(K) \ge 2,
\end{cases}
\]
где $\Delta = k/2 - m$.
\end{claim}
\begin{claim}\label{conj_class_6}
Для всякой конъюнкции $K\in\mathbb{K}_6$ выполнено
\[
\begin{cases}
\left|\Theta^0 \cap N_K\right| \le k/2 \\
\left|\Theta^1 \cap N_K\right| \le k/2+\Delta \\
rank^{-}(K) \ge 1 \\
rank^{+}(K) \ge 1,
\end{cases}
\]
где $\Delta = k/2 - m$.
\end{claim}

Обозначим $|\mathbb{K}_1| = \mu_1, |\mathbb{K}_2| = \mu_2, ..., |\mathbb{K}_6| = \mu_6$. Основой последующих рассуждений служит
\begin{lemma}\label{lemma_card_size_inequality}
Для всякой ДНФ реализации $D$ функции $f\in\Phi^{n,k}_{m}$ выполнены неравенства
\begin{gather}
(1+3\varepsilon)\mu_2 + 2\varepsilon\mu_1 + \mu_6 + \mu_4 + \mu_3 > n(1-\varepsilon)\label{inequality_1_3} \\
(1+\varepsilon)\mu_1 + (1+\varepsilon)\mu_6 +2 \mu_5 +\varepsilon \mu_4 + 2\varepsilon\mu_3 > n\label{inequality_2_3_4},
\end{gather}
при $\varepsilon = 2\Delta/k$, $\Delta = k/2 - m$.
\end{lemma}
\begin{proof}
Заметим, что $\Theta^1 \ge (k/2 - \Delta) n$, с другой стороны в соответствии с утверждениями \ref{conj_class_1}--\ref{conj_class_6} оценивается сверху следующим образом
\[(k/2 + 3\Delta) + 2 \Delta \mu_1 + k\mu_6/2 + k\mu_4/2 + k\mu_3/2 > \Theta^1 \ge (k/2 - \Delta) n\]
Поделив обе части неравенства на $k/2$, получим неравенство \ref{inequality_1_3}.

Аналогично оценим $\Theta^0$:
\[(k/2+\Delta)\mu_1 + (k/2 +\Delta)\mu_6 + (k-1)\mu_5 + \Delta\mu_4 + 2\Delta\mu_3 > \Theta^0 \ge n\]
Из полученного неравенства следует неравенство \ref{inequality_2_3_4}.
\end{proof}

\begin{lemma}\label{lemma_rank_inequality}
Для произвольной ДНФ реализации $D$ функции $f$ класса $\Phi^{n,k}_m$~справедливы неравенства
\begin{gather}
rank^{+}\; D \ge \max\{2n-\mu_1,\; \mu_1 + 2\mu_5 + \mu_6 + 2\mu_4 + \mu_3\}\label{inequality_rank_plus}\\
rank^{-}\; D \ge \max\{2n - \mu_2 - \mu_4 - \mu_3,\; 3\mu_2 + 2\mu_1 + \mu_6 + \mu_4 + 2\mu_3\}\label{inequality_rank_minus}
\end{gather}
\end{lemma}
\begin{proof}
Заметим, что при заданном разбиении на классы $\mathbb{K}_1$ -- $\mathbb{K}_6$ число собственных отрицательных литералов равно $\mu_1$. По леммам \ref{literal_min_number_a} и \ref{literal_min_number_b} все несобственные литералы встречаются как минимум два раза, следовательно, $rank^{+} \ge \mu_1 + 2 (n-\mu_1)$. Кроме того, число положительных литералов, в силу утверждений \ref{conj_class_1}--\ref{conj_class_6} оценивается снизу выражением $\mu_1 + 2\mu_5 + \mu_6 + 2\mu_4 + \mu_3$.

Таким же образом получается оценка в неравенстве \ref{inequality_rank_plus}.
\end{proof}

\begin{theorema}\label{theorem_main}
Ранг ДНФ реализации $D$ функции $f$ класса $\Phi^{n,k}_m$ ограничен снизу выражением
\[rank\;D > \frac{10}{3}n - \frac{5n\varepsilon}{3(1+\varepsilon)},\;\text{при}\quad\varepsilon = 2\frac{(k/2 - m)}{k} \le \frac{1}{4}\]
\[rank\;D > \frac{10}{3}n - \frac{13n\varepsilon}{9 + 3\varepsilon},\;\text{при}\quad\frac{1}{4} < \varepsilon < \frac{1}{3}.\]
\end{theorema}
\begin{proof}
Для доказательства утверждения теоремы воспользуемся леммами \ref{lemma_card_size_inequality} и \ref{lemma_rank_inequality}. Обозначим

\begin{multline*}
L(D) := \max\{2n-\mu_1,\; \mu_1 + 2\mu_5 + \mu_6 + 2\mu_4 + \mu_3\} + \\ \max\{2n - \mu_2 - \mu_4 - \mu_3,\; 3\mu_2 + 2\mu_1 + \mu_6 + \mu_4 + 2\mu_3\}.
\end{multline*}

Далее нас будет интересовать минимальное значение указанной функции при ограничениях, определяемых леммой \ref{lemma_card_size_inequality}. Для доказательства теоремы рассмотрим следующие случаи:
\begin{enumerate}
\item[$a.$] В условиях рассматриваемого случая:

    $\begin{cases}
        2n - \mu_1 \ge \mu_1 + 2\mu_5 + \mu_6 + 2\mu_4 + \mu_3\\
        2n - \mu_2 - \mu_4 - \mu_3 \ge 3\mu_2 + 2\mu_1 + \mu_6 + \mu_4 + 2\mu_3
    \end{cases}$

    Вместе с неравенствами леммы \ref{lemma_card_size_inequality}, получаем следующую систему неравенств
\[    \begin{cases}
&     \mu_1 + \mu_6/2 + \mu_5 + \mu_4 + \mu_3/2 \le n\\
&    \mu_2 + \mu_1/2 + \mu_6/4 + \mu_4/2 + 3\mu_3/4 \le n/2\\
&    -(1+3\varepsilon)\mu_2 - 2\varepsilon\mu_1 - \mu_6 - 2\mu_4 - \mu_3 < -n(1 - \varepsilon)\\
&    -(1+\varepsilon)\mu_1 - (1+\varepsilon)\mu_6 - 2\mu_5 - \varepsilon\mu_4 - 2\varepsilon\mu_3 < -n
    \end{cases}\]

    Первое неравенство системы умножим на 2, второе на 4, последнее на $(1-\varepsilon)$; сложив с третьим, получим
    \begin{multline*}
    3(1-\varepsilon)(\mu_2 + \mu_1 + \mu_4 + \mu_3) + 2\varepsilon\mu_5 + \varepsilon^2 \mu_6 + (\varepsilon + \varepsilon^2)\mu_1 +\\
    (2\varepsilon + \varepsilon^2)\mu_4 + (\varepsilon + 2\varepsilon^2)\mu_3 < 2 (1 + \varepsilon)n
    \end{multline*}
    Откуда
    \[\mu_2 + \mu_1 + \mu_4 + \mu_3 < \frac{2}{3}n \left(1 + 2\frac{\varepsilon}{1 - \varepsilon}\right)\]
    Следовательно $L(D) = 4n - \mu_2 -\mu_1 - \mu_4 - \mu_3 > \frac{10}{3}n - \frac{4n\varepsilon}{3(1 - \varepsilon)}$.
\item[$b.$]
В условиях рассматриваемого случая:

    $\begin{cases}
        2n - \mu_1 \ge \mu_1 + 2\mu_5 + \mu_6 + 2\mu_4 + \mu_3\\
        2n - \mu_2 - \mu_4 - \mu_3 <  3\mu_2 + 2\mu_1 + \mu_6 + \mu_4 + 2\mu_3
    \end{cases}$

    Оценим функцию $L(D)$:
    \[L(D) \ge 2n + 3\mu_2 + \mu_1 + \mu_6 + \mu_4 + 2\mu_3.\]
    Воспользовавшись неравенствами \ref{inequality_1_3}--\ref{inequality_2_3_4}, получим
    \[    \begin{cases}
&    \mu_1 + \mu_6/2 + \mu_5 + \mu_4 + \mu_3/2 \le n\\
&    -\mu_2 - \mu_1/2 - \mu_6/4 - \mu_4/2 - 3\mu_3/4 \le -n/2\\
&    -(1+3\varepsilon)\mu_2 - 2\varepsilon\mu_1 - \mu_6 - 2\mu_4 - \mu_3 < -n(1 - \varepsilon)\\
&    -(1+\varepsilon)\mu_1 - (1+\varepsilon)\mu_6 - 2\mu_5 - \varepsilon\mu_4 - 2\varepsilon\mu_3 < -n
    \end{cases}\]
    Откуда
    \[
    3(1+\varepsilon)\mu_2 + (1+\varepsilon)\mu_1 + (1+\varepsilon)\mu_6 + (1+\varepsilon)\mu_4 + 2(1+\varepsilon)\mu_3 >   \frac{n}{3}(4-\varepsilon).
    \]

    Следовательно $L(D) > \frac{10}{3}n - \frac{5n\varepsilon}{3(1+\varepsilon)}$.
\item[$c.$]
    $\begin{cases}
        2n - \mu_1 < \mu_1 + 2\mu_5 + \mu_6 + 2\mu_4 + \mu_3\\
        2n - \mu_2 - \mu_4 - \mu_3 \ge  3\mu_2 + 2\mu_1 + \mu_6 + \mu_4 + 2\mu_3
    \end{cases}$

    Оценим функцию $L(D)$:
    \[L(D) \ge 2n - \mu_2 + \mu_1 + \mu_6 + 2\mu_5 + 2\mu_4 + \mu_3\]

    Из определяющей максимум системы неравенств, получаем
    \begin{gather*}
        \mu_5 + \mu_4/2 > 2\mu_2 + \mu_3/2,\\
        3(1-\varepsilon)\mu_2/2 + \varepsilon\mu_1 + \mu_6/2 + \mu_4/2 + 2\mu_3/2 > n(1-\varepsilon)/2,\\
        (1-\varepsilon)(\mu_1 + \mu_5 + \mu_6/2 + \mu_4 + \mu_3/2) > n(1-\varepsilon).
    \end{gather*}
    Откуда $L(D) > 2n + 3n(1-\varepsilon)/2 + \mu_2(1+\varepsilon)/2 + \mu_3(1+2\varepsilon)/2 > \frac{7}{2}n - \frac{3n\varepsilon}{2}$.
\item[$d.$]
    $\begin{cases}
        2n - \mu_1 < \mu_1 + 2\mu_5 + \mu_6 + 2\mu_4 + \mu_3\\
        2n - \mu_2 - \mu_4 - \mu_3 < 3\mu_2 + 2\mu_1 + \mu_6 + \mu_4 + 2\mu_3
    \end{cases}$

    Оценим функцию $L(D)$:
    \[L(D) \ge 3\mu_2 + 3\mu_1 + 2\mu_6 + 2\mu_5 + 3\mu_4 + 3\mu_3.\]

    Система неравенств в рассматриваемом случае имеет вид
\[    \begin{cases}
&    -\mu_1 - \mu_6/2 - \mu_5 - \mu_4 - \mu_3/2 \le -n\\
&   -\mu_2 - \mu_1/2 - \mu_6/4 - \mu_4/2 - 3\mu_3/4 \le -n/2\\
&    -(1+3\varepsilon)\mu_2 - 2\varepsilon\mu_1 - \mu_6 - 2\mu_4 - \mu_3 < -n(1 - \varepsilon)\\
&    -(1+\varepsilon)\mu_1 - (1+\varepsilon)\mu_6 - 2\mu_5 - \varepsilon\mu_4 - 2\varepsilon\mu_3 < -n
    \end{cases}\]

    Умножим первое неравенство на 4, второе на 8 и сложим с последними двумя
    \begin{multline}3(1 + \varepsilon/3)\mu_2 + 3(1 + \varepsilon/3)\mu_1 + 2(1 + \varepsilon/6)\mu_6 + \\ 2\mu_5 + 3(1 + \varepsilon/9)\mu_4 + 3(1 + 2\varepsilon/9)\mu_3 > \frac{10 - n\varepsilon}{3}
    \end{multline}
    Следовательно, $L(D) > \frac{10}{3}n - \frac{13n\varepsilon}{9 + 3\varepsilon}$.
\end{enumerate}
Заметим, что оценки в случаях $b$ и $d$ совпадают при $\varepsilon = 1/4$ и равны $3n$. При $\varepsilon < 1/4$ оценка случая $b$ является наименьшей, что завершает доказательство теоремы.
\end{proof}

\begin{theorema}\label{bound_lower}
Минимальная ДНФ почти всех приведенных функций из $\Phi^{m,k}$ при $\log m \le k/32$ содержит не менее
\[
\frac{10}{3}m - \frac{5m\left(1 - \alpha\sqrt{\frac{2\log m}{k}}\right)}{3+3\alpha\sqrt{\frac{2\log m}{k}}}
\]
литералов, а при $k/162 < \log m < k/32$ число литералов не меньше
\[
\frac{10}{3}m - \frac{13m\left(1 - \alpha\sqrt{\frac{2\log m}{k}}\right)}{9+3\alpha\sqrt{\frac{2\log m}{k}}}
\]
где $\alpha$ произвольное положительное число меньшее 1, $\log x$ обозначен натуральный логарифм числа $x$.
\end{theorema}
\begin{proof}

Выберем $\lambda = \alpha\sqrt{\frac{2\log m}{k}}$ при $\alpha < 1$. Воспользуемся неравенством Чернова
\[
    \sum^{k/2-\lambda}_{t=0}{k\choose t}\leqslant 2^ke^{-2\lambda^2/k}.
\]

Математическое ожидание числа столбцов веса меньше $k/2 - \lambda$ в случайной функции $f\in\Phi^{n,k}$
\[
E[\text{числа столбцов веса меньше $k/2 - \lambda$}] \le m 2^{k+1} e^{-2\lambda^2/k} = o(1).
\]

Из неравенства Маркова для случайной функции $f$ следует
\[
Pr\left[f \not\in \Phi^{m,k}_{k/2 - \alpha\sqrt{\frac{2\log m}{k}}} \mid f \in \Phi^{m,k}\right] = o(1)
\]

Применение предыдущей теоремы завершает доказательство.
\end{proof}

\begin{example}
Рассмотрим булеву функцию $h$ заданную матрицей нулей $H_k$ с максимальным числом различных (с точностью до отрицания) столбцов, каждый из которых имеет вес не менее $k/2 - \sqrt{k\ln k}$. В соответствии с неравенством Чернова число столбцов в такой матрице не меньше $2^{k-1}(1 - 2/k)$. По теореме \ref{theorem_main} число литералов в минимальной ДНФ такой функции строго больше
\[2^{k-1}\left(\frac{10}{3} - \frac{5}{3}\frac{2\ln k}{\sqrt{k} + 2\ln k}\right)\left(1 - \frac{2}{k}\right) = \frac{10}{3}\cdot2^{k-1}(1+o(1))\]
Таким образом, число литералов в минимальной ДНФ указанной функции существенно больше, чем в минимальной ДНФ полной функции с тем же числом нулей.
\end{example}

\section{Заключение}
В настоящей работе получены нижние оценки сложности ДНФ булевых функций, задаваемых матрицами нулей, не содержащими столбцы с малым числом единиц (нулей).

Полученные оценки позволяют строить высокие нижние оценки для большого числа приведенных функций высокой размерности с малым числом нулей.

В работе показано, что ДНФ полной булевой функции не является асимптотически максимальной по сложности в своем классе; тем самым решен один из наиболее важных открытых вопросов теории сложности ДНФ булевых функций с малым числом нулей.

\bibliographystyle{chicago}
\bibliography{dnflow}

\end{document}